\documentclass[reqno,11pt]{amsart}
\usepackage{graphicx}
\usepackage{amsmath,amsthm,amsfonts,amssymb,latexsym,epsfig,geometry,amsaddr}
\usepackage{lineno,hyperref}

\newtheorem{theorem}{Theorem}[section]

\newtheorem{conj}{Conjecture}[section]
\theoremstyle{definition}

\numberwithin{equation}{section}
\newgeometry{left=3cm,right=3cm,top=3cm,bottom=3cm}

\begin{document}

\title[short text for running head]{``Cubic Partitions''}
\title{On 3 and 9-regular cubic partitions}

\author{D. S. Gireesh$^1$,M. S. Mahadeva Naika$^2$ and Shivashankar C.$^1$}
\address{$^1$Department of Mathematics, M. S. Ramaiah University of Applied Sciences, Peenya, Bengaluru-560 058, Karnataka, India}
\address{$^2$Department of Mathematics, Bangalore University, Central College Campus, Bengaluru-560 001, Karnataka, India}

\email{$^1$gireeshdap@gmail.com; shankars224@gmail.com; $^2$msmnaika@rediffmail.com}
          
\begin{abstract}
Let $a_3(n)$ and $a_9(n)$ are 3 and 9-regular cubic partitions of $n$. In this paper, we find the infinite 
family of congruences modulo powers of 3 for $a_3(n)$ and $a_9(n)$ such as
\[a_3\left (3^{2\alpha}n+\frac{3^{2\alpha}-1}{4}\right )\equiv 0 \pmod{3^{\alpha}}\] and
\[a_9\left (3^{\alpha+1}n+3^{\alpha+1}-1\right )\equiv 0 \pmod{3^{\alpha+1}}.\]
\end{abstract}


\maketitle

\section{Introduction}\label{intro}
A partition of a positive integer $n$ is a non-increasing sequence of positive integers whose sum is $n$. Let $p(n)$ denotes the number of partitions of $n$ and the generating function is
\[\sum\limits_{n\geq0}p(n)q^n=\frac1{f_1},\]
where, here and throughout the paper, we set
\[f_k=(q^k;q^k)_\infty=\prod\limits_{m=1}^{\infty}(1-q^{km}).\]
Chan \cite{HCC} studied the cubic partition function denoted by $a(n)$, whose generating function is
\[\sum\limits_{n\geq0}a(n)q^n=\frac1{f_1f_2}.\]
He found the generating function
\[\sum\limits_{n\geq0}a(3n+2)q^n=3\frac{f_3^3f_6^3}{f_1^4f_2^4}\]
which readily implies that
\[a(3n+2)\equiv 0\pmod{3}.\]
Chan \cite{HCC1} also established infinite family of congruence modulo powers of 3 for $a(n)$. For each $n, k\geq1$, Chan proved that
\begin{equation}
a(3^kn+c_k)\equiv0\pmod{3^{k+\delta(k)}},
\end{equation}
where $c_k$ is the reciprocal modulo $3^k$ of 8 and
\[\delta(k):=\begin{cases}
 \text{1 if $k$ is even}\\
\text{0 if $k$ is odd}.
\end{cases}
\]
Zhao and Zhong \cite{ZZ}  have studied cubic partition pairs denoted by $b(n)$ and the generating function satisfied by $b(n)$ is
\begin{equation}\label{deff122}
\sum\limits_{n\geq 0}b(n)q^n=\frac1{f_1^2f_2^2}.
\end{equation}
For each $n\geq 0$, they found the Ramanujan's type congruences
\begin{align*}
b(5n+4)&\equiv 0\pmod{5},\\
b(7n+i)&\equiv 0\pmod{7},\\
b(9n+7)&\equiv 0\pmod{9},
\end{align*}
where $i\in\{2, 3, 4, 6\}$.

Lin \cite{Lin} has also studied the cubic partition pairs and established  the following Ramanujan's type congruences modulo 27:
\begin{align}
b(27n+16)&\equiv 0\pmod{27},\label{LC1}\\
b(27n+25)&\equiv 0\pmod{27},\label{LC2}\\
b(81n+61)&\equiv 0\pmod{27}.\label{LC3}
\end{align}
He also proposed the following conjectures:
\begin{conj} For each $n\geq 0$,
	\begin{equation}\label{8161}
	b(81n+61)\equiv 0\pmod{81}.
	\end{equation}
\end{conj}

\begin{conj}
	\begin{align}
	\sum\limits_{n\geq 0}b(81n+7)q^n&\equiv 9\frac{f_2f_3^2}{f_6}\pmod{81},\\
	\sum\limits_{n\geq 0}b(81n+34)q^n&\equiv 36\frac{f_1f_6^2}{f_3}\pmod{81}.
	\end{align}
\end{conj}
Gireesh and Naika have conformed Lin's conjectures in \cite{GM} and Chern also proved Lin's conjectures in \cite{Chern}.

Motivated by the above results, in this paper, we study 3 and 9-regular cubic partitions, which are defined as follows:

\noindent$\bullet$ Let $a_3(n)$ denotes the number of 3-regular cubic partitions of $n$, whose generating function is
\begin{equation}\label{defa3}
\sum\limits_{n\geq 0}a_3(n)q^n=\frac{f_3f_6}{f_1f_2}.
\end{equation}
$\bullet$ Let $a_9(n)$ denotes the number of 9-regular cubic partitions of $n$, whose generating function is
\begin{equation}\label{defa9}
\sum\limits_{n\geq0}a_9(n)q^n=\frac{f_9f_{18}}{f_1f_2}.
\end{equation}

We will show that
\begin{equation}\label{3p31}
\sum\limits_{n\geq 0}a_3(3n+2)q^n=3\frac{f_3^3f_6^3}{f_1^3f_2^3}
\end{equation}
and
\begin{equation}\label{3a32}
\sum\limits_{n\geq 0}a_9(3n+2)q^n=3\frac{f_3^4f_6^4}{f_1^4f_2^4}.
\end{equation}
These are analogous to Ramanujan's most beautiful identities \cite[p. 239 and p. 243]{RAMAU}
\begin{equation}\label{Rama54}
\sum\limits_{n\geq 0}p(5n+4)q^n=5\frac{f_5^5}{f_1^6}
\end{equation}
and
\begin{equation}\label{Rama75}
\sum\limits_{n\geq 0}p(7n+5)q^n=7\frac{f_7^3}{f_1^4}+49q\frac{f_7^7}{f_1^8}.
\end{equation}
We also find the infinite family of congruences modulo powers of 3 for $a_3(n)$ and $a_9(n)$, which are stated in the following theorems:
\begin{theorem}\label{th1}
For each $\alpha \geq0$,
\begin{align}
a_3\left (3^{2\alpha}n+\frac{3^{2\alpha}-1}{4}\right )&\equiv 0 \pmod{3^{\alpha}},\label{3ac1}\\
a_3\left (3^{2\alpha+1}n+\frac{3^{2\alpha+2}-1}{4}\right )&\equiv 0 \pmod{3^{\alpha+1}},\label{3ac2}\\
a_3\left (3^{2\alpha+2}n+\frac{7\times 3^{2\alpha+1}-1}{4}\right )&\equiv 0 \pmod{3^{\alpha+2}},\label{3ac3}\\
a_3\left (3^{2\alpha+2}n+\frac{11\times 3^{2\alpha+1}-1}{4}\right )&\equiv 0 \pmod{3^{\alpha+2}}.\label{3ac4}
\end{align}
\end{theorem}

\begin{theorem}\label{th2}
For each $\alpha \geq0$,
\begin{equation}\label{9ac}
a_9\left (3^{\alpha+1}n+3^{\alpha+1}-1\right )\equiv 0 \pmod{3^{\alpha+1}}.
\end{equation}
\end{theorem}

The results \eqref{3ac1}--\eqref{9ac} are  analogous to Ramanujan's congruences modulo powers of 5 \cite{Hir}, for $\alpha\geq0$ and for $n\geq0$,
\begin{equation}
p\left (5^{2\alpha+1}n+\frac{19\times 5^{2\alpha+1}+1}{24}\right )\equiv 0 \pmod{5^{2\alpha+1}}
\end{equation}
and
\begin{equation}
p\left (5^{2\alpha+2}n+\frac{23\times 5^{2\alpha+2}+1}{24}\right )\equiv 0 \pmod{5^{2\alpha+2}}.
\end{equation}

\section{Preliminaries}
Due to Chan \cite{HCC}, we have the following identities
\begin{equation}\label{f12}
f_1f_2=f_9f_{18}\left(\frac1{x(q^3)}-q-2q^2x(q^3)\right)
\end{equation}
and
\begin{equation}
\frac{f_3^4f_6^4}{f_9^4f_{18}^4}=\frac1{x(q^3)^3}-7q^3-8q^6x(q^3)^3,
\end{equation}
where 
\[x(q):=q^{-1/3}\left({\cfrac{q^{1/3}}{1}}_+{\cfrac{q+q^2}{1}}_+{\cfrac{q^2+q^4}{1}}_+ \dots\right).\]

Now let
\begin{equation}\label{p3}
\zeta=\frac{f_1f_2}{qf_9f_{18}},\ \ \rho=\frac{1}{qx(q^3)},\ \ 
T=\frac{f_3^4f_6^4}{q^3f_9^4f_{18}^4}.
\end{equation}

Then, from \eqref{f12}--\eqref{p3},
\begin{equation}\label{z}
\zeta=\frac{f_1f_2}{qf_9f_{18}}=\rho-1-\frac{2}{\rho}
\end{equation}
and
\begin{equation}\label{T}
T=\rho^3-7-\frac{8}{\rho^3}.
\end{equation}
From \eqref{z} and \eqref{T}, we have
\begin{align}\label{z3}
\zeta^3&=\rho^3-3\rho^2-3\rho+11+\frac{6}{\rho}-\frac{12}{\rho^2}-\frac{8}{\rho^3}\nonumber\\
&=T+18-3\rho^2-3\rho+\frac{6}{\rho}-\frac{12}{\rho^2}\nonumber\\
&=T+9-3\zeta^2-9\rho+\frac{18}{\rho}\nonumber\\
&=T-9\zeta-3\zeta^2.
\end{align}

It follows from \eqref{z3} that
\begin{equation}\label{z31}
\zeta^3+3\zeta^2+9\zeta=T.
\end{equation}

We can write \eqref{z31}
\begin{equation}
\frac1{\zeta}=\frac1{T}(9+3\zeta+\zeta^2),
\end{equation}
so 
\begin{equation}\label{zi}
\frac1{\zeta^i}=\frac1{T}\left (\frac{9}{\zeta^{i-1}}+\frac{3}{\zeta^{i-2}}+\frac1{\zeta^{i-3}}\right ).
\end{equation}

Now let $H$ be the ``huffing'' operator modulo 3, that is,
\[H\left (\sum a_nq^n\right )=\sum a_{3n}q^{3n}.\]

If we apply $H$ to \eqref{zi}, we find
\begin{equation}\label{hzi}
H\left (\frac1{\zeta^i}\right )=\frac1T\left (9H\left (\frac1{\zeta^{i-1}}\right )
+3H\left (\frac1{\zeta^{i-2}}\right )+H\left (\frac1{\zeta^{i-3}}\right )\right ).
\end{equation}

Now,
\begin{eqnarray}
&&H\left (\zeta^2\right )=H\left (\rho^2-2\rho-3+\frac{4}{\rho}+\frac{4}{\rho^2}\right )=-3,\\
&&H\left (\zeta\right )=H\left (\rho-1-\frac{2}{\rho}\right )=-1,\\
&&H\left (1\right )=1.\label{h1}
\end{eqnarray}
From \eqref{hzi}--\eqref{h1}, we find
\begin{equation}
H\left (\frac1\zeta\right )=\frac{3}{T},\label{Hiz}
\end{equation}
\begin{equation}
H\left (\frac1{\zeta^2}\right )=\frac{2}{T}+\frac{3^3}{T^2},
\end{equation}
\begin{equation}
H\left (\frac1{\zeta^3}\right )=\frac1{T}+\frac{3^3}{T^2}+\frac{3^5}{T^3},
\end{equation}
and so on.

Indeed, for $i\ge1$ we can write
\begin{equation}
H\left (\frac1{\zeta^i}\right )=\sum_{j=1}^i\frac{m_{i,j}}{T^j},
\end{equation}
where the $m_{i,j}$ are defined in the following matrix.

The $m_{i,j}$ form a matrix $M$, the first nine rows of which are
\begin{equation}\label{matrix}
M=\left (\begin{matrix} 3&0&0&0&0&0&0&\cdots\\
2&3^3&0&0&0&0&0&\cdots\\
1&3^3&3^5&0&0&0&0&\cdots\\
0&2\cdot 3^2&2^2\cdot 3^4&3^7&0&0&0&\cdots\\
0&5&2\cdot 3^3\cdot 5&3^6\cdot 5&3^9&0&0&\cdots\\
0&1&2\cdot 3^2\cdot 7 &3^6\cdot 5&2\cdot 3^9&3^{11}&0&\cdots\\
0&0&2\cdot 3\cdot 7&2^2\cdot 3^4 \cdot 7&3^8\cdot 7&3^{10}\cdot 7&3^{13}&\cdots\\
0&0&2^3&2\cdot 3^3 \cdot 19&2^4\cdot3^7&2^2\cdot 3^9\cdot 7&2^3\cdot 3^{12}&\cdots\\
0&0&1&2^2\cdot 3^4&3^9&3^9\cdot 5^2&2^2\cdot 3^{13}&\cdots\\
\vdots&\vdots&\vdots&\vdots&\vdots&\vdots&\vdots
\end{matrix}\right )\\
\end{equation}
\begin{equation}\label{mij}
\text{and for}\  i\ge4,\ m_{i,1}=0,\ \text{and for}\ j\ge2,\ 
m_{i,j}=9m_{i-1,j-1}+3m_{i-2,j-1}+m_{i-3,j-1}.
\end{equation}
In fact $m_{4i-3,j}=0$ for $j\leq i-1$, so we can write
\begin{equation}\label{Hz1}
H\left (\frac1{\zeta^{4i-3}}\right )=\sum_{j=i}^{4i-3}\frac{m_{4i-3,j}}{T^j}
=\sum_{j=1}^{3i-2}\frac{m_{4i-3,i+j-1}}{T^{i+j-1}}
=\sum_{j=1}^{3i-2}\frac{a_{i,j}}{T^{i+j-1}},
\end{equation}
where 
\begin{equation}\label{aij}
a_{i,j}=m_{4i-3,i+j-1}.
\end{equation}

Similarly, $m_{4i-1,j}=0$ if $j\leq i-1$, so we can write
\begin{equation}\label{Hz2}
H\left (\frac1{\zeta^{4i-1}}\right )
=\sum_{j=i}^{4i-1}\frac{m_{4i-1,j}}{T^j}\\
=\sum_{j=1}^{3i}\frac{m_{4i-1,i+j-1}}{T^{i+j-1}}
=\sum_{j=1}^{3i}\frac{b_{i,j}}{T^{i+j-1}},
\end{equation}
where \begin{equation}\label{bij}
b_{i,j}=m_{4i-1,i+j-1}.
\end{equation}

And $m_{4i,j}=0$ if $j\leq i$, so we can write
\begin{equation}\label{Hz3}
H\left (\frac1{\zeta^{4i}}\right )
=\sum_{j=1+i}^{4i}\frac{m_{4i,j}}{T^j}\\
=\sum_{j=1}^{3i}\frac{m_{4i,i+j}}{T^{i+j}}
=\sum_{j=1}^{3i}\frac{c_{i,j}}{T^{i+j}},
\end{equation}
where 
\begin{equation}\label{cij}
c_{i,j}=m_{4i,i+j}.
\end{equation}

We can write \eqref{Hz1}
\begin{equation}
H\left (\left (q\frac{f_9f_{18}}{f_1f_2}\right )^{4i-3}\right )
=\sum_{j=1}^{3i-2}a_{i,j}\left (q^3\frac{f_9^4f_{18}^4}{f_3^4f_6^4}\right )^{i+j-1},
\end{equation}
and this can be rearranged to
\begin{equation}
H\left (q^{i-3}\left(\frac{f_3f_6}{f_1f_2}\right )^{4i-3}\right )
=\sum_{j=1}^{3i-2}a_{i,j}q^{3j-3}\left (\frac{f_9f_{18}}{f_3f_6}\right )^{4j-1}.
\end{equation}

The equation \eqref{Hz2} can be written as
\begin{equation}
H\left (\left (q\frac{f_9f_{18}}{f_1f_2}\right )^{4i-1}\right )
=\sum_{j=1}^{3i}b_{i,j}\left (q^3\frac{f_9^4f_{18}^4}{f_3^4f_6^4}\right )^{i+j-1},
\end{equation}
and this can be rearranged to
\begin{equation}
H\left (q^{i-1}\left(\frac{f_3f_6}{f_1f_2}\right )^{4i-1}\right )
=\sum_{j=1}^{3i}b_{i,j}q^{3j-3}\left (\frac{f_9f_{18}}{f_3f_6}\right )^{4j-3}.
\end{equation}

Similarly \eqref{Hz3} is
\begin{equation}
H\left (\left (q\frac{f_9f_{18}}{f_1f_2}\right )^{4i}\right )
=\sum_{j=1}^{3i}c_{i,j}\left (q^3\frac{f_9^4f_{18}^4}{f_3^4f_6^4}\right )^{i+j},
\end{equation}
and this can be rearranged to
\begin{equation}
H\left (q^i\left(\frac{f_3f_6}{f_1f_2}\right )^{4i}\right )
=\sum_{j=1}^{3i}c_{i,j}q^{3j}\left (\frac{f_9f_{18}}{f_3f_6}\right )^{4j}.
\end{equation}

\section{generating functions}
In this section, we found some generating functions which are useful in proving our main results.
\begin{theorem}
For each  $\alpha\geq 0$,
\begin{equation}\label{3p33}
\sum_{n\ge0}a_3\left(3^{2\alpha}n+\frac{3^{2\alpha}-1}{4}\right )q^n
=\sum_{i\geq1}x_{2\alpha,i}q^{i-1}\left(\frac{f_3f_6}{f_1f_2}\right)^{4i-3}
\end{equation}
and
\begin{equation}\label{3p34}
\sum_{n\ge0}a_3\left (3^{2\alpha+1}n+\frac{3^{2\alpha+2}-1}{4}\right )q^n
=\sum_{i\geq1}x_{2\alpha+1,i}q^{i-1}\left(\frac{f_3f_6}{f_1f_2}\right)^{4i-1},
\end{equation}
where the coefficient vectors ${\mathbf{x}}_{\alpha}=(x_{\alpha,1},x_{\alpha,2},\ \dots\ )$
are given by
\begin{equation}\label{x0}
\mathbf{x}_0=(x_{0,1},x_{0,2},x_{0,3},\ \dots\ )=(1,0,0,\ \dots\ ),
\end{equation}
and
\begin{eqnarray}
&&{\mathbf{x}}_{\alpha+1}={\mathbf{x}}_{\alpha}A\ \ \text{if}\ \alpha\ \text{is even},\\
&&{\mathbf{x}}_{\alpha+1}={\mathbf{x}}_{\alpha}B\ \ \text{if}\ \alpha\ \text{is odd},
\end{eqnarray}
where $A=(a_{i,j})_{i,j\ge1}$ and $B=(b_{i,j})_{i,j\ge1}$.
\end{theorem}
\begin{proof}
The identity \eqref{defa3} is the $\alpha=0$ case of \eqref{3p33}.\newline
Suppose \eqref{3p33} holds for some $\alpha\geq 0$. Then
\begin{equation}
\sum_{n\ge0}a_3\left(3^{2\alpha}n+\frac{3^{2\alpha}-1}{4}\right )q^n
=\sum_{i\geq1}x_{2\alpha,i}q^{i-1}\left(\frac{f_3f_6}{f_1f_2}\right)^{4i-3},
\end{equation}
which is equivalent to
\begin{equation}\label{3p39}
\sum_{n\ge0}a_3\left(3^{2\alpha}n+\frac{3^{2\alpha}-1}{4}\right )q^{n-2}
=\sum_{i\geq1}x_{2\alpha,i}q^{i-3}\left(\frac{f_3f_6}{f_1f_2}\right)^{4i-3}.
\end{equation}
Applying the operator $H$ to \eqref{3p39}, we find that
\begin{align*}
\sum_{n\ge0}a_3\left(3^{2\alpha}(3n+2)+\frac{3^{2\alpha}-1}{4}\right )q^{3n}
&=\sum_{i\geq1}x_{2\alpha,i}H\left(q^{i-3}\left(\frac{f_3f_6}{f_1f_2}\right)^{4i-3}\right)\\
&=\sum_{i\geq1}x_{2\alpha,i}\sum\limits_{j=1}^{3i-2}a_{i,j}q^{3j-3}\left(\frac{f_9f_{18}}{f_3f_6}\right)^{4j-1}\\
&=\sum\limits_{j\geq1}\left(\sum_{i\geq1}x_{2\alpha,i}a_{i,j}\right)q^{3j-3}\left(\frac{f_9f_{18}}{f_3f_6}\right)^{4j-1}\\
&=\sum\limits_{j\geq1}x_{2\alpha+1,j}q^{3j-3}\left(\frac{f_9f_{18}}{f_3f_6}\right)^{4j-1},
\end{align*}
which implies that
\begin{equation}
\sum_{n\ge0}a_3\left(3^{2\alpha+1}n+\frac{3^{2\alpha+2}-1}{4}\right )q^{n}=\sum\limits_{j\geq1}x_{2\alpha+1,j}q^{j-1}\left(\frac{f_3f_6}{f_1f_2}\right)^{4j-1},
\end{equation}
which is \eqref{3p34}.

Now suppose \eqref{3p34} holds for some $\alpha\geq0$. Then
\begin{equation}\label{3p310}
\sum_{n\ge0}a_3\left (3^{2\alpha+1}n+\frac{3^{2\alpha+2}-1}{4}\right )q^n
=\sum_{i\geq1}x_{2\alpha+1,i}q^{i-1}\left(\frac{f_3f_6}{f_1f_2}\right)^{4i-1}.
\end{equation}
Applying the operator $H$ to \eqref{3p310}, we find that
\begin{align*}
\sum_{n\ge0}a_3\left (3^{2\alpha+1}(3n)+\frac{3^{2\alpha+2}-1}{4}\right )q^{3n}
&=\sum_{i\geq1}x_{2\alpha+1,i}H\left(q^{i-1}\left(\frac{f_3f_6}{f_1f_2}\right)^{4i-1}\right)\\
&=\sum_{i\geq1}x_{2\alpha+1,i}\sum_{j=1}^{3i}b_{i,j}q^{3j-3}\left(\frac{f_9f_{18}}{f_3f_6}\right)^{4j-3}\\
&=\sum_{j\geq1}\left(\sum_{i\geq1}x_{2\alpha+1,i}b_{i,j}\right)q^{3j-3}\left(\frac{f_9f_{18}}{f_3f_6}\right)^{4j-3}\\
&=\sum_{j\geq1}x_{2\alpha+2,j}q^{3j-3}\left(\frac{f_9f_{18}}{f_3f_6}\right)^{4j-3}.
\end{align*}
After simplification, we obtain
\begin{equation}
\sum_{n\ge0}a_3\left (3^{2\alpha+2}n+\frac{3^{2\alpha+2}-1}{4}\right )q^{n}=\sum_{j\geq1}x_{2\alpha+2,j}q^{j-1}\left(\frac{f_3f_6}{f_1f_2}\right)^{4j-3},
\end{equation}
which is \eqref{3p33} with $\alpha+1$ in place of $\alpha$. This completes the proof of \eqref{3p33} and \eqref{3p34} by induction.
\end{proof}

\begin{theorem}For each $\alpha\geq0$,
\begin{equation}\label{3a9}
\sum_{n\ge0}a_9\left(3^{\alpha+1}n+3^{\alpha+1}-1\right )q^n
=\sum_{i\geq1}y_{\alpha,i}q^{i-1}\left(\frac{f_3f_6}{f_1f_2}\right)^{4i}
\end{equation}
where the coefficient vectors ${\mathbf{Y}}_{\alpha}=(y_{\alpha,1},y_{\alpha,2},\ \dots\ )$
are given by
\begin{equation}\label{y0}
\mathbf{Y}_0=(y_{0,1},y_{0,2},y_{0,3},\ \dots\ )=(3,0,0,\ \dots\ ),
\end{equation}
and
\begin{eqnarray}
{\mathbf{Y}}_{\alpha+1}={\mathbf{Y}}_{\alpha}C,
\end{eqnarray}
where $C=(c_{i,j})_{i,j\ge1}$.
\end{theorem}
\begin{proof}
The identity \eqref{3a32} is the $\alpha=0$ case of \eqref{3a9}.\newline
Suppose \eqref{3a9} holds for some $\alpha\geq 0$. Then
\begin{equation}
\sum_{n\ge0}a_9\left(3^{\alpha+1}n+3^{\alpha+1}-1\right )q^n
=\sum_{i\geq1}y_{\alpha,i}q^{i-1}\left(\frac{f_3f_6}{f_1f_2}\right)^{4i},
\end{equation}
which is equivalent to
\begin{equation}\label{3p9}
\sum_{n\ge0}a_9\left(3^{\alpha+1}n+3^{\alpha+1}-1\right )q^{n-2}
=q^{-3}\sum_{i\geq1}y_{\alpha,i}q^i\left(\frac{f_3f_6}{f_1f_2}\right)^{4i}.
\end{equation}
Applying the operator $H$ to \eqref{3p9}, we find that
\begin{align*}
\sum_{n\ge0}a_9\left(3^{\alpha+1}(3n+2)+3^{\alpha+1}-1\right)q^{3n}
&=q^{-3}\sum_{i\geq1}y_{\alpha,i}H\left(q^i\left(\frac{f_3f_6}{f_1f_2}\right)^{4i}\right)\\
&=\sum_{i\geq1}y_{\alpha,i}\sum\limits_{j=1}^{3i}c_{i,j}q^{3j-3}\left(\frac{f_9f_{18}}{f_3f_6}\right)^{4j}\\
&=\sum\limits_{j\geq1}\left(\sum_{i\geq1}y_{\alpha,i}c_{i,j}\right)q^{3j-3}\left(\frac{f_9f_{18}}{f_3f_6}\right)^{4j}\\
&=\sum\limits_{j\geq1}y_{\alpha+1,j}q^{3j-3}\left(\frac{f_9f_{18}}{f_3f_6}\right)^{4j},
\end{align*}
which implies that
\begin{equation}
\sum_{n\ge0}a_9\left(3^{\alpha+2}n+3^{\alpha+2}-1\right )q^n=\sum\limits_{j\geq1}y_{\alpha+1,j}q^{j-1}\left(\frac{f_3f_6}{f_1f_2}\right)^{4j},
\end{equation}
which is \eqref{3a9} with $\alpha+1$ for $\alpha$ .
\end{proof}

\section{Congruences}
Let $\nu(N)$ be the largest power of 3 that divides $N$. Note that $\nu(0)=+\infty$.

\noindent\textbf{Proof of the Theorem \ref{th1}.}
It follows from \eqref{matrix} and \eqref{mij} that
\begin{equation}\label{numij}
\nu(m_{i.j})\geq 3j-i-1,
\end{equation}
and then follows from \eqref{aij}, \eqref{bij} and \eqref{numij} that
\begin{equation}\label{nuaij}
\nu(a_{i.j})\geq 3(i+j-1)-(4i-3)-1=3j-i-1
\end{equation}
and
\begin{equation}\label{nubij}
\nu(b_{i.j})\geq 3(i+j-1)-(4i-1)-1=3j-i-3.
\end{equation}
It not hard to show that
\begin{equation}\label{nux1}
\nu(x_{2\alpha,j})\geq \alpha+3j-4
\end{equation}
and
\begin{equation}\label{nux2}
\nu(x_{2\alpha+1,j})\geq \alpha+1+3(j-1).
\end{equation}
The identity \eqref{nux1} is true for $\alpha=0$, by \eqref{x0}.\newline
Suppose \eqref{nux1} is true for some $\alpha\geq0$. Then
\begin{align*}
\nu(x_{2\alpha+1,j})&\geq \underset{i\geq 1}{\text{min}}\left(\nu(x_{2\alpha,i})+\nu(a_{i,j})\right)\\
&=\nu(x_{2\alpha,1})+\nu(a_{1,j})\\
&\geq \alpha+3j-2\\
&\geq \alpha+1+3(j-1),
\end{align*}
which is \eqref{nux2}.

Now suppose \eqref{nux2} is true for all $\alpha\geq0$. Then
\begin{align*}
\nu(x_{2\alpha+2,j})&\geq \underset{i\geq 1}{\text{min}}\left(\nu(x_{2\alpha+1,i})+\nu(b_{i,j})\right)\\
&=\nu(x_{2\alpha+1,1})+\nu(b_{1,j})\\
&\geq \alpha+1+3j-4,
\end{align*}
which is \eqref{nux1} with $\alpha+1$ in place of $\alpha$. This completes the proof of \eqref{nux1} and \eqref{nux2}  by induction.

The congruence \eqref{3ac1} follows from \eqref{3p33} together with \eqref{nux1}, and the congruence \eqref{3ac2} follows from \eqref{3p34} together with \eqref{nux2}.

It follows from \eqref{3p34} and \eqref{nux2} that
\begin{equation}\label{3c1}
\sum\limits_{n\geq 0}a_3\left(3^{2\alpha+1}n+\frac{3^{2\alpha+2}-1}{4}\right)q^n\equiv 3^{\alpha+1}\frac{f_3^3f_6^3}{f_1^3f_2^3}\pmod{3^{\alpha+4}}.
\end{equation}
By the binomial theorem, it is easy to see that
\begin{equation}\label{fp3}
f_1^3\equiv f_3\pmod{3}.
\end{equation}
In view of \eqref{fp3}, the congruence \eqref{3c1} can be expressed as
\begin{align}
\sum\limits_{n\geq 0}a_3\left(3^{2\alpha+1}n+\frac{3^{2\alpha+2}-1}{4}\right)q^n\equiv 3^{\alpha+1}\frac{f_9f_{18}}{f_3f_6}\pmod{3^{\alpha+2}}.\label{3c31}
\end{align}
Equating the coefficients of $q^{3n+1}$ and $q^{3n+2}$ in \eqref{3c31}, we obtain \eqref{3ac3} and \eqref{3ac4}, respectively.

\vspace{0.5cm}
\noindent\textbf{Proof of the Theorem \ref{th2}.} It follows from \eqref{cij} and \eqref{numij} that
\begin{equation}\label{nucij}
\nu(c_{i.j})\geq 3(i+j)-4i-1=3j-i-1.
\end{equation}
It not hard to show that
\begin{equation}\label{nuy1}
\nu(y_{\alpha,j})\geq \alpha+1+3(j-1).
\end{equation}
The identity \eqref{nuy1} is true for $\alpha=0$, by \eqref{y0}.\newline
Suppose \eqref{nuy1} is true for some $\alpha\geq0$. Then
\begin{align*}
\nu(y_{\alpha+1,j})&\geq \underset{i\geq 1}{\text{min}}\left(\nu(y_{\alpha,i})+\nu(c_{i,j})\right)\\
&=\nu(y_{\alpha,1})+\nu(c_{1,j})\\
&\geq \alpha+1+3j-2\\
&\geq \alpha+2+3(j-1),
\end{align*}
which is \eqref{nuy1} with $\alpha+1$ for $\alpha$.

The congruence \eqref{9ac} follows from \eqref{3a9} together with \eqref{nuy1}.


\end{document}